\documentclass{amsart}

\newtheorem{theorem}{Theorem}[section]
\newtheorem{lemma}[theorem]{Lemma}
\newtheorem{corollary}[theorem]{Corollary}
\newtheorem{proposition}[theorem]{Proposition}
\newtheorem{remark}[theorem]{Remark}

\theoremstyle{definition}

\theoremstyle{remark}

\numberwithin{equation}{section}

\begin{document}
\title[Ternary Kloosterman Sums with Square Arguments]{$\begin{array}{c}
         \text{Infinite Families of Recursive Formulas}\\
         \text{Generating Power Moments of Ternary Kloosterman Sums}\\
           \text{ with Square Arguments Associated with $O^{-}_{}(2n,q)$}
       \end{array}$}

\author{dae san kim}
\address{Department of Mathematics, Sogang University, Seoul 121-742, Korea}
\curraddr{Department of Mathematics, Sogang University, Seoul
121-742, Korea} \email{dskim@sogong.ac.kr}
\thanks{}

\subjclass[]{}

\date{}

\dedicatory{ }


\begin{abstract}
In this paper, we construct eight infinite families of ternary
linear codes  associated with double cosets with respect to certain
maximal parabolic subgroup  of the special orthogonal group
$SO^{-}(2n,q)$. Here ${q}$ is a power of three. Then we obtain four
infinite families of recursive formulas for power moments of
Kloosterman sums with square arguments and four infinite families of
recursive formulas for even power moments of those in terms of the
frequencies of weights in the codes. This is done via Pless power
moment identity and by utilizing the explicit expressions of
exponential sums over those double cosets  related to the
evaluations of $\lq\lq$Gauss sums" for the orthogonal  groups
$O^{-}(2n,q)$.\\

  Index terms- Kloosterman sum, orthogonal  group, special orthogonal group, double cosets,
maximal parabolic subgroup, Pless power moment identity, weight
distribution.\\


\end{abstract}

\maketitle

\section{Introduction}

 Let $\psi$ be a nontrivial additive character of the finite field
 $\mathbb{F}_{q}$ with $q=p^r$ elements ($p$ a prime). Then the
Kloosterman sum $K(\psi;a)$ (\cite{LN1}) is defined by

\begin{align*}
K(\psi;a)=\sum_{\alpha\in\mathbb{F}_{q}^{*}}\psi(\alpha_{}^{}+a\alpha_{}^{-1})
(a\in\mathbb{F}_{q}^{*}).
\end{align*}
For this, we have the Weil bound
\begin{equation}\label{a}
|K(\psi;a)|\leq 2\sqrt{q}.
\end{equation}
The Kloosterman sum was introduced in 1926(\cite{K1}) to give an
estimate for the Fourier coefficients of modular forms.

For each nonnegative integer $h$, by $MK(\psi)^{h}$ we will denote
the $h$-th moment of the Kloosterman sum $K(\psi;a)$. Namely, it is
given by
\begin{align*}
MK(\psi)^{h}=\sum_{a\in\mathbb{F}_{q}^{*}}K(\psi;a)^{h}.
\end{align*}
If $\psi=\lambda$ is the canonical additive character of
$\mathbb{F}_{q}$, then $MK(\lambda)^{h}$ will be simply denoted by
$MK^{h}$.

Also, we introduce an incomplete power moments of Kloosterman sums.
Namely, for every nonnegative integer $h$, and $\psi$ as before, we
define
\begin{equation}\label{b}
SK(\psi)^{h}=\sum_{a\in\mathbb{F}_{q}^{*},~a~square}K(\psi;a)^{h},
\end{equation}
which is called the $h$-th moment of Kloosterman sums with
$\lq\lq$square arguments". If $\psi=\lambda$ is the canonical
additive character of $\mathbb{F}_q$, then $SK(\lambda)^{h}$ will be
denoted by $SK^h$, for brevity.

Explicit computations on power moments of Kloosterman sums were
begun with the paper \cite{S1} of Sali\'{e} in 1931, where he
showed, for any odd prime $q$,
\begin{align*}
MK^h=q^2 M_{h-1} -(q-1)^{h-1}+2(-1)^{h-1}~(h\geq1).
\end{align*}
Here $M_0=0$, and, for $h\in\mathbb{Z}_{>0}$,
\begin{align*}
M_h=|\{(\alpha_1,\cdots,\alpha_h)\in(\mathbb{F}_q^*)^h|\sum_{j=1}^h
\alpha_j=1=\sum_{j=1}^h \alpha_j^{-1}\}|.
\end{align*}
For $q=p$ odd prime, Sali\'{e} obtained $MK^1$, $MK^2$, $MK^3$,
$MK^4$ in \cite{S1} by determining $M_1$, $M_2$, $M_3$. On the other
hand, $MK^5$ can be expressed in terms of the $p$-th eigenvalue for
a weight 3 newform on $\Gamma_0(15)$ (cf. \cite{L1}, \cite{PV1}).
$MK^6$ can be expressed in terms of the $p$-th eigenvalue for a
weight 4 newform on $\Gamma_0(6)$ (cf. \cite{HS1}). Also, based on
numerical evidence, in \cite{E1} Evans was led to propose a
conjecture which expresses $MK^7$ in terms of Hecke eigenvalues for
a weight 3 newform on $\Gamma_0(525)$ with quartic nebentypus of
conductor 105.

Assume from now on that $q=3^r$. Recently, Moisio was able to find
explicit expressions of $MK^h$, for $h\leq10$ (cf.\cite{M1}). This
was done, via Pless power moment identity, by connecting moments of
Kloosterman sums and the frequencies of weights in the ternary Melas
code of length $q-1$, which were known by the work of Geer, Schoof
and Vlugt in \cite{GS1}. In \cite{D6}, we were able to produce two
recursive formulas generating power moments of Kloosterman sums with
square arguments and one recursive formula generating even power
moments of those. To do that, we constructed three ternary linear
codes $C(SO^{-}(2,q))$, $C(O^{-}(2,q))$, $C(SO^{-}(4,q))$,
respectively associated with the orthogonal groups $SO^{-}(2,q)$,
$O^{-}(2,q)$, $SO^{-}(4,q)$, and express those power moments in
terms of the frequencies of weights in each code.  In \cite{DJ1},
the symplectic groups $Sp(2,q)$ and $Sp(4,q)$ were used instead in
order to produce recursive formulas generating power moments  and
even power moments of Kloosterman sums with square arguments.

  In this paper, we will be able to produce four infinite families of
recursive formulas generating power moments of Kloosterman sums with
square arguments and four infinite families of recursive formulas
generating even power moments of  those. To do that, we construct
eight  infinite families of ternary linear codes $C(DC_1^+(n,q))$
$(n=2,4,\cdots )$,  $C(DC_1^-(n,q))$ $(n=1,3,\cdots)$,  both
associated with  $Q\sigma_{n-1}Q$; $C(DC_2^+(n,q))$ $(n=2,4,
\cdots)$,  $C(DC_2^-(n,q))$ $(n=3,5,\cdots)$  both associated with
$Q\sigma_{n-2}Q$; $C(DC_3^+(n,q))$ $(n=2,4, \cdots)$,
$C(DC_3^-(n,q))$ $(n=3,5,\cdots)$  both associated with  $\rho
Q\sigma_{n-2}Q$; $C(DC_4^+(n,q))$ $(n=4,6,\cdots)$, $C(DC_4^-(n,q))$
$(n=3,5,\cdots)$  both associated with  $\rho Q\sigma_{n-3}Q$, with
respect to the maximal parabolic subgroup $Q=Q(2n,q)$ of the special
orthogonal group $SO^-(2n,q)$, and express those power moments in
terms of the frequencies of weights in each code. Then, thanks to
our previous results on the explicit expressions of exponential sums
over those double cosets related to the evaluations of $\lq\lq$Gauss
sums" for the orthogonal  groups $O^-(2n,q)$ $[4,5]$, we can express
the weight of each codeword in the duals of the codes in terms of
Kloosterman sums or squares of Kloosterman sums. Then our formulas
will follow immediately from the Pless power moment identity.
Analogously to these, in \cite{D5}, we obtained infinite families of
recursive formulas for power moments of Kloosterman sums with square
arguments and for even power moments of those by constructing
ternary linear codes associated with double cosets with respect to
certain maximal parabolic subgroup of the special orthogonal group
$SO^-(2n,q)$.

Theorem \ref{A} in the following (cf. (\ref{s}), (\ref{t}),
(\ref{v})-(\ref{y})) is the main result of this paper. Henceforth,
we agree that, for nonnegative integers $a$, $b$, $c$,

\begin{equation*}
{\binom{c}{a,b}}={\frac{c!}{a!~b!~(c-a-b)!}},~if~a+b\leq c,
\end{equation*}

and

\begin{equation*}\label{d}
{\binom{c}{a,b}}=0,~if~a+b>c.
\end{equation*}

To simplify notations, we introduce the following ones which will be
used throughout this paper at various places.

\begin{equation}\label{c}
A_1^+(n,q)=q^{\frac{1}{4}(5n^2-2n-4)}(q^{n-1}-1)\prod_{j=1}^{(n-2)/2}(q^{2j-1}-1),
\end{equation}

\begin{equation}\label{d}
B_1^+(n,q)=(q+1)q^{\frac{1}{4}n^2}\prod _{j=1}^{(n-2)/2}(q^{2j}-1),
\end{equation}

\begin{equation}\label{e}
A_2^+(n,q)=q^{\frac{1}{4}(5n^2-2n-8)}{\begin{bmatrix}
                                        n-1 \\
                                           1 \\
                                      \end{bmatrix}
 }_q\prod_{j=1}^{(n-2)/2}
(q^{2j-1}-1),
\end{equation}

\begin{equation}\label{f}
B_2^+(n,q)=(q+1)q^{\frac{1}{4}(n-2)^2}(q^{n-1}-1)\prod_{j=1}^{(n-2)/2}
(q^{2j}-1),
\end{equation}

\begin{equation}\label{g}
A_3^+(n,q)=(q+1)q^{\frac{1}{4}(5n^2-2n-8)}{\begin{bmatrix}
                                            n-1 \\
                                           1 \\
                                           \end{bmatrix}
 }_q\prod_{j=1}^{(n-2)/2}
(q^{2j-1}-1),
\end{equation}

\begin{equation}\label{h}
B_3^+(n,q)=q^{\frac{1}{4}(n-2)^2}(q^{n-1}-1)\prod _{j=1}^{(n-2)/2}
(q^{2j}-1),
\end{equation}

\begin{equation}\label{i}
A_4^+(n,q)=(q+1)q^{\frac{1}{4}(5n^2-6n-4)}{\begin{bmatrix}
                                             n-1 \\
                                           2 \\
                                           \end{bmatrix}
 }_q
\prod_{j=1}^{(n-2)/2} (q^{2j-1}-1),
\end{equation}

\begin{equation}\label{j}
B_4^+(n,q)=q^{{\frac{1}{4}(n-2)^2}}(q^{n-1}-1)\prod _{j=1}^{(n-2)/2}
(q^{2j}-1),
\end{equation}

\begin{equation}\label{k}
A_1^-(n,q)=q^{\frac{5}{4}(n^2-1)}\prod _{j=1}^{(n-1)/2}
(q^{2j-1}-1),
\end{equation}

\begin{equation}\label{l}
B_1^-(n,q)=(q+1)q^{\frac{1}{4}(n-1)^2}\prod_{j=1}^{(n-1)/2}
(q^{2j}-1),
\end{equation}

\begin{equation}\label{m}
A_2^-(n,q)=q^{\frac{1}{4}(5n^2-4n-5)}{\begin{bmatrix}
                                        n-1 \\
                                        1 \\
                                      \end{bmatrix}
}_q\prod_{j=1}^{(n-1)/2}
(q^{2j-1}-1),
\end{equation}

\begin{equation}\label{n}
B_2^-(n,q)=(q+1)q^{\frac{1}{4}(n-1)^2}\prod_{j=1}^{(n-1)/2}
(q^{2j}-1),
\end{equation}

\begin{equation}\label{o}
A_3^-(n,q)=(q+1)q^{\frac{1}{4}(5n^2-4n-5)}{\begin{bmatrix}
                                        n-1 \\
                                        1 \\
                                      \end{bmatrix}}_q\prod_{j=1}^{(n-1)/2}
(q^{2j-1}-1),
\end{equation}

\begin{equation}\label{p}
B_3^-(n,q)=q^{\frac{1}{4}(n-1)^2}\prod _{j=1}^{(n-1)/2}
(q^{2j}-1),
\end{equation}

\begin{equation}\label{q}
A_4^-(n,q)=(q+1)q^{\frac{1}{4}(5n^2-4n-9)}{\begin{bmatrix}
                                        n-1 \\
                                        2 \\
                                      \end{bmatrix}}_q\prod_{j=1}^{(n-3)/2}
(q^{2j-1}-1),
\end{equation}

\begin{equation}\label{r}
B_4^-(n,q)=q^{\frac{1}{4}(n-3)^2}(q^{n-2}-1)(q^{n-1}-1)\prod_{j=1}^{(n-3)/2}
(q^{2j}-1).
\end{equation}

From now on, it is assumed that either $+$signs or $-$signs are
chosen everywhere, whenever $\pm$ signs appear.

\begin{theorem}\label{A}

Let $q=3^r$. Then with the notations in (\ref{c})-(\ref{r}), we have
the following.

\item\label{Aa}
$(1)$ With  $i=1,3,$ and $+$  signs everywhere for $\pm$ signs,  we
have a recursive formula generating power moments of Kloosterman
sums with square arguments over $\mathbb{F}_q$ (cf. (\ref{b})) for
each
 $n\geq2$ even and all $q$; with $i=1$ and $-$ signs everywhere for
 $\pm$ signs, we have such a fomula, for each $n\geq1$ odd and all
 $q$; with $i=3$ and $-$ signs everywhere for $\pm$ signs,  we have such a formula,
 for each $n\geq3$ odd and all $q$.

\begin{equation}\label{s}
\begin{split}&(\pm(-1))^h SK_{}^{h}=-\sum_{l=0}^{h-1}(\pm((-1))^l{\binom{h}{l}}
B_i^{\pm}(n,q)^{h-l}SK_{}^{l}\\
            &~\quad\quad
            +qA_i^{\pm}(n,q)^{-h}\sum_{j=0}^{min\{N_i^{\pm}(n,q),h\}}(-1)^{j}C_{i,j}^{\pm}(n,q)\\
            &~\quad\quad\times\sum_{t=j}^{h}t!S(h,t)3^{h-t}2^{t-h-j-1}{\binom{N_i^{\pm}(n,q)
            -j}{N_i^{\pm}(n,q)-t}}~(h=1,2,\cdots),
\end{split}
\end{equation}
where
$n_i^{\pm}(n,q)=|DC_i^{\pm}(n,q)|=A_i^{\pm}(n,q)B_i^{\pm}(n,q)$, and
$\{C_{i,j}^{\pm}(n,q)\}_{j=0}^{N_i^{\pm}(n,q)}$ is the weight
distribution of the ternary linear code $C(DC_i^{\pm}(n,q))$ given
by

\begin{equation*}
C_{i,j}^{\pm}(n,q)=
\sum\binom{q^{-1}A_i^{\pm}(n,q)(B_i^{\pm}(n,q)\pm1)}{\nu_1,\mu_1}
\binom{q^{-1}A_i^{\pm}(n,q)(B_i^{\pm}(n,q)\pm1)}{\nu_{-1},\mu_{-1}}
\end{equation*}
\begin{equation}\label{t}
\begin{split}
&\qquad\quad\times\prod_{\beta^2-1\neq0 ~square}\binom{q^{-1}A_i^{\pm}(n,q)(B_i^{\pm}(n,q)\pm(q+1))}{\nu_\beta,\mu_\beta}\\
&\qquad\quad\times\prod_{\beta^2-1
~nonsquare}\binom{q^{-1}A_i^{\pm}(n,q)(B_i^{\pm}(n,q)\pm(-q+1))}{\nu_\beta,\mu_\beta},
\end{split}
\end{equation}
with the sum running over all the sets of nonnegative integers
$\{\nu_\beta\}_{\beta\in\mathbb{F}_{q}}$ and
$\{\mu_\beta\}_{\beta\in\mathbb{F}_{q}}$ satisfying

\begin{align*}
\sum_{\beta\in\mathbb{F}_q} \nu_\beta
+\sum_{\beta\in\mathbb{F}_q}\mu_\beta=j,\quad and \quad
\sum_{\beta\in\mathbb{F}_q} \nu_\beta
\beta=\sum_{\beta\in\mathbb{F}_q}\mu_\beta \beta.
\end{align*}

In addition, $S(h,t)$  is the Stirling number of the second kind
defined by
\begin{equation}\label{u}
S(h,t)=\frac{1}{t!}\sum_{j=0}^{t}(-1)^{t-j}{\binom{t}{j}}j^h.
\end{equation}
\item\label{Ab}
$(2)$ With $+$ signs everywhere for $\pm$ signs,  we have recursive
formulas generating
 even power moments of Kloosterman sums with square arguments over $\mathbb{F}_{q}$,
 for each $n\geq2$ even and all $q$; with $-$ signs everywhere for $\pm$ signs,
 we have such a formula, for each $n\geq3$ odd and  all $q$.

\begin{equation}\label{v}
\begin{split}&(\pm1)^{h}SK_{}^{2h}=-\sum_{l=0}^{h-1}(\pm1)^{l}{\binom{h}{l}}B_2^{\pm}(n,q)^{h-l}SK_{}^{2l}\\
            &~\quad\quad
            +qA_2^{\pm}(n,q)^{-h}\sum_{j=0}^{min\{N_2^{\pm}(n,q),h\}}(-1)^{j}C_{2,j}^{\pm}(n,q)\\
            &\qquad\quad\quad\quad\times\sum_{t=j}^{h}t!S(h,t)3^{h-t}2^{t-h-j-1}{\binom{N_2^{\pm}(n,q)
            -j}{N_2^{\pm}(n,q)-t}}~ (h=1,2,\cdots),
\end{split}
\end{equation}
where
$n_2^{\pm}(n,q)=|DC_2^{\pm}(n,q)|=A_2^{\pm}(n,q)B_2^{\pm}(n,q)$, and
$\{C_{2,j}^{\pm}(n,q)\}_{j=0}^{N_2^{\pm}(n,q)}$ is the weight
distribution of the ternary linear code $C(DC_2^{\pm}(n,q))$ given
by
\begin{equation}\label{w}
C_{2,j}^{\pm}(n,q)= \sum\prod_{\beta\in
\mathbb{F}_{q}}\binom{q^{-1}A_2^{\pm}(n,q)(B_2^{\pm}(n,q){\pm}((q-1)^2-q\delta(2,q;\beta)))}{\nu_\beta,\mu_\beta},
\end{equation}
with the sum running over all the sets of nonnegative integers
$\{\nu_\beta\}_{\beta\in\mathbb{F}_{q}}$ and
$\{\mu_\beta\}_{\beta\in\mathbb{F}_{q}}$ satisfying

\begin{align*}
\sum_{\beta\in\mathbb{F}_q} \nu_\beta
+\sum_{\beta\in\mathbb{F}_q}\mu_\beta=j,\quad and \quad
\sum_{\beta\in\mathbb{F}_q} \nu_\beta
\beta=\sum_{\beta\in\mathbb{F}_q}\mu_\beta \beta,
\end{align*}
and $\delta(2,q;\beta)=|\{(\alpha_1,\alpha_2)\in\mathbb{F}_{q}^{2}|
\alpha_1+\alpha_2+\alpha_1^{-1}+\alpha_2^{-1}=\beta\}|$.

\item\label{Ac}
$(3)$ With $+$ signs everywhere for $\pm$ signs,  we have recursive
formulas generating
 even power moments of Kloosterman sums with square arguments over $\mathbb{F}_{q}$,
 for each $n\geq4$ even and all $q$; with $-$ signs everywhere for $\pm$ signs,
 we have such a formula, for each $n\geq3$ odd and  all $q$.

\begin{equation}\label{x}
\begin{split}
&(\pm1)^{h}SK_{}^{2h}=-\sum_{l=0}^{h-1}(\pm1)^{l}{\binom{h}{l}}
\{B_{4}^{\pm}(n,q){\pm}(q^2-q)\}^{h-l}SK_{}^{2l}\\
&~\quad\quad
+qA_4^{\pm}(n,q)^{-h}\sum_{j=0}^{min\{N_4^{\pm}(n,q),h\}}(-1)^{j}C_{4,j}^{\pm}(n,q)\\
&\qquad\quad\quad\quad\times\sum_{t=j}^{h}t!S(h,t)3^{h-t}2^{t-h-j-1}{\binom{N_4^{\pm}(n,q)
-j}{N_4^{\pm}(n,q)-t}}~ (h=1,2,\cdots),
\end{split}
\end{equation}
where
$n_4^{\pm}(n,q)=|DC_4^{\pm}(n,q)|=A_4^{\pm}(n,q)B_4^{\pm}(n,q)$, and
$\{C_{4,j}^{\pm}(n,q)\}_{j=0}^{N_4^{\pm}(n,q)}$ is the weight
distribution of the ternary linear code $C(DC_4^{\pm}(n,q))$ given
by

\begin{equation*}
C_{4,j}^{\pm}(n,q)=
\sum\binom{q^{-1}A_4^{\pm}(n,q)(B_4^{\pm}(n,q)\pm(-1)(q\delta(2,q;\beta)+(q-1)^3))}{\nu_0,\mu_0}
\end{equation*}
\begin{equation}\label{y}
\qquad\quad\times\prod_{\beta\neq0}\binom{q^{-1}A_4^{\pm}(n,q)(B_4^{\pm}(n,q)\pm(-1)(q\delta(2,q;\beta)-2q^2+3q-1))}{\nu_\beta,\mu_\beta},
\end{equation}
with the sum running over all the sets of nonnegative integers
$\{\nu_\beta\}_{\beta\in\mathbb{F}_{q}}$ and
$\{\mu_\beta\}_{\beta\in\mathbb{F}_{q}}$ satisfying

\begin{align*}
\sum_{\beta\in\mathbb{F}_q} \nu_\beta
+\sum_{\beta\in\mathbb{F}_q}\mu_\beta=j,\quad and \quad
\sum_{\beta\in\mathbb{F}_q} \nu_\beta
\beta=\sum_{\beta\in\mathbb{F}_q}\mu_\beta \beta,
\end{align*}

\end{theorem}


\section{$O^{-}(2n,q)$}

For more details about this section, one is referred to the paper
\cite{D1} and \cite{D2}. Throughout this paper, the following
notations will be used:
\begin{align*}
\begin{split}
q&=3^r~(r\in\mathbb{Z}_{>0}),\\
\mathbb{F}_q&=~the~finite~field~ with~ q~ elements,\\
TrA&=~the~ trace~ of~ A~ for~ a~ square~ matrix~ A,\\
^t B&=~the~transpose~ of~ B~for~any~matrix~B.
\end{split}
\end{align*}

The orthogonal group $O^-(2n,q)$ over the field $\mathbb{F}_q$ is
defined as:
\begin{align*}
O^-(2n,q)=\{w\in GL(2n,q)|^twJw=J\},
\end{align*}
where
\begin{align*}
J=\begin{bmatrix}
    0 & 1_{n-1} & 0 & 0 \\
    1_{n-1} & 0 & 0 & 0 \\
    0 & 0 & 1 & 0 \\
    0 & 0 & 0 & -\epsilon \\
  \end{bmatrix}
,
\end{align*}
and $\epsilon$ is a fixed element in $\mathbb{F}_q^*\setminus
{\mathbb{F}_q^*}^2$, here and throughout this paper. \\
For convenience, we put
\begin{equation}\label{z}
\delta_\epsilon=\begin{bmatrix}
                  1 & 0 \\
                  0 & -\epsilon \\
                \end{bmatrix}
.
 \end{equation}
Then $O^-(2n,q)$ consists of all matrices

\begin{align*}
\begin{bmatrix}
  A & B & e \\
  C & D & f \\
  g & h & i \\
\end{bmatrix}
\begin{split}
(A,B,C,D~(n-&1)\times(n-1),e,f~(n-1)\times2,\\
&g,h~2\times(n-1),i~2\times2)
\end{split}
\end{align*}
in $GL(2n,q)$ satisfying the relations:

\begin{align*}
\begin{split}
&{}^tAC+{}^tCA+{}^tg\delta_\epsilon g=0,\\
&{}^tBD+{}^tDB+{}^th\delta_\epsilon h=0,\\
&{}^tef+{}^tfe+{}^ti\delta_\epsilon i=\delta_\epsilon,\\
&{}^tAD+{}^tCB+{}^tg\delta_\epsilon h=1_{n-1},\\
&{}^tAf+{}^tCe+{}^tg\delta_\epsilon i=0,\\
&{}^tBf+{}^tDe+{}^th\delta_\epsilon i=0.\\
\end{split}
\end{align*}

The special orthogonal group $SO^-(2n,q)$ over the field
$\mathbb{F}_q$ is defined as:

\begin{align*}
SO^-(2n,q)=\{w\in O^-(2n,q)|det~ w=1\},
\end{align*}
which is a subgroup of index 2 in $O^-(2n,q)$.

In particular, we have
\begin{align*}
O^-(2,q)&=\{i\in GL(2,q)|^ti\delta_\epsilon i=\delta_\epsilon\}
\end{align*}
\begin{equation}\label{a1}
\qquad \qquad \qquad \qquad \qquad =SO^-(2,q)\amalg\begin{bmatrix}
                                                     1 & 0 \\
                                                     0 & -1 \\
                                                    \end{bmatrix}
SO^-(2,q),~
\end{equation}
with
\begin{align*}
\begin{split}
SO^-(2,q)&=\left\{\begin{bmatrix}
               a & b\epsilon \\
               b & a \\
             \end{bmatrix}
\Bigg| ~a,b\in\mathbb{F}_q,~a^2-b^2 \epsilon =1\right\}\\
&=\left\{\begin{bmatrix}
     a & b\epsilon \\
     b & a \\
    \end{bmatrix}
\Bigg|~a+b\epsilon\in\mathbb{F}_q(\epsilon)~with~N_{\mathbb{F}_q(\epsilon)/\mathbb{F}_q}(a+b\epsilon)=1\right\}.
\end{split}
\end{align*}

Let $P(2n,q)$ be the maximal parabolic subgroup of $O^-(2n,q)$ given
by
\begin{align*}
\begin{split}
P&=P(2n,q)\\
&=\left\{\begin{bmatrix}
       A & 0 & 0 \\
        0 & ^tA^{-1} & 0 \\
        0 & 0 & i \\
    \end{bmatrix}
\begin{bmatrix}
  1_{n-1} & B & -^th\delta_\epsilon \\
               0 & 1_{n-1} & 0 \\
               0 & h & 1_2 \\
\end{bmatrix}
\Bigg|\begin{array}{c}
   A\in GL(n-1,q) \\
   i\in O^-(2,q) \\
   ^tB+B+^th\delta_\epsilon h=0
 \end{array}
\right\},
\end{split}
\end{align*}
and let $Q=Q(2n,q)$ be the subgroup of $P(2n,q)$ of index 2 defined
by
\begin{align*}
\begin{split}
Q&=Q(2n,q)\\
&=\left\{\begin{bmatrix}
       A & 0 & 0 \\
        0 & ^tA^{-1} & 0 \\
        0 & 0 & i \\
    \end{bmatrix}
\begin{bmatrix}
  1_{n-1} & B & -^th\delta_\epsilon \\
               0 & 1_{n-1} & 0 \\
               0 & h & 1_2 \\
\end{bmatrix}
\Bigg|~\begin{array}{c}
   A\in GL(n-1,q) \\
   i\in SO^-(2,q) \\
   ^tB+B+^th\delta_\epsilon h=0
 \end{array}
\right\}.
\end{split}
\end{align*}
From (\ref{a1}), we see that
\begin{equation}\label{b1}
P=Q\amalg\rho Q,
\end{equation}
with
\begin{align*}
\rho=\begin{bmatrix}
      1_{n-1} & 0 & 0 & 0 \\
         0 & 1_{n-1} & 0 & 0 \\
         0 & 0 & 1 & 0 \\
         0 & 0 & 0 & -1 \\
     \end{bmatrix}
 .
\end{align*}

Let $\sigma_r$ denote the following matrix in $O^{-}(2n,q)$

\begin{align*}
\sigma_r=\begin{bmatrix}
           0 & 0 & 1_r & 0 & 0 \\
             0 & 1_{n-1-r} & 0 & 0 & 0 \\
             1_r & 0 & 0 & 0 & 0 \\
             0 & 0 & 0 & 1_{n-1-r} & 0 \\
             0 & 0 & 0 & 0 & 1_2 \\
         \end{bmatrix}
 ~(0\leq r\leq n-1).
\end{align*}
Then the Bruhat decomposition of $O^-(2n,q)$ with respect to
$P=P(2n,q)$ is given by
\begin{align}\label{c1}
O^-(2n,q)=\coprod_{r=0}^{n-1}P\sigma_r
P=\coprod_{r=0}^{n-1}P\sigma_r Q,
\end{align}
which can further be modified as
\begin{equation}\label{d1}
\begin{split}
O^-(2n,q)&=\coprod_{r=0}^{n-1}P\sigma_r (B_r\setminus Q) \\
~\quad\quad&=\coprod_{r=0}^{n-1}Q\sigma_r (B_r\setminus Q)\amalg
\coprod_{r=0}^{n-1}(\rho Q)\sigma_r (B_r\setminus Q),
\end{split}
\end{equation}
with
\begin{align*}
B_r=B_r (q)=\{w\in Q(2n,q)|\sigma_r w \sigma_r^{-1}\in P\}.
\end{align*}

The order of the general linear group $GL(n,q)$ is given by
\begin{align*}
g_n=\prod_{j=0}^{n-1}(q^{n}-q^{j})=q^{\binom{n}{2}}\prod_{j=1}^{n}(q^j-1).
\end{align*}
For integers $n,r$ with $0\leq r\leq n$, the $q$-binomial
coefficients are defined as:

\begin{align*}
\begin{bmatrix}
  n \\
  r \\
\end{bmatrix}
_q=\prod_{j=0}^{r-1}(q^{n-j}-1)/(q^{r-j}-1).
\end{align*}

Then one can show that
\begin{align*}
|P(2n,q)|=2(q+1)g_{n-1}q^{(n-1)(n+2)/2},
\end{align*}
\begin{align}\label{e1}
|B_r\setminus Q|={\Big[
                    \begin{array}{c}
                      n-1 \\
                      r \\
                    \end{array}
                  \Big]
 }_q q^{r(r+3)/2}(0\leq r\leq n-1)~
\end{align}
\qquad \qquad \qquad \qquad \qquad \qquad \qquad \qquad \qquad
\qquad (cf. \cite{D1}, (3.12), (3.20), (3.21)),

\begin{equation}\label{f1}
\begin{split}
|Q(2n,q)\sigma_r Q(2n,q)|&=|\rho Q(2n,q)\sigma_r Q(2n,q)|\\
                         &=\frac{1}{2}|P(2n,q)\sigma_r Q(2n,q)|\\
                         &=\frac{1}{2}|P(2n,q)||B_r \setminus Q(2n,q)|\\
                         &=(q+1)q^{n^2-n}\prod_{j=1}^{n-1}(q^j-1){\Big[
                    \begin{array}{c}
                      n-1 \\
                      r \\
                    \end{array}
                  \Big]}_q q^{\binom{r}{2}
                         }q^{2r}.
\end{split}
\end{equation}

Let
\begin{equation}\label{g1}
DC_{1}^{+}(n,q)=Q(2n,q)\sigma_{n-1}Q(2n,q), \quad
for~n=2,4,6,\cdots,
\end{equation}
\begin{equation}\label{h1}
DC_{2}^{+}(n,q)=Q(2n,q)\sigma_{n-2}Q(2n,q), \quad
for~n=2,4,6,\cdots,
\end{equation}
\begin{equation}\label{i1}
DC_{3}^{+}(n,q)=\rho Q(2n,q)\sigma_{n-2}Q(2n,q), \quad
for~n=2,4,6,\cdots,
\end{equation}
\begin{equation}\label{j1}
DC_{4}^{+}(n,q)=\rho Q(2n,q)\sigma_{n-3}Q(2n,q), \quad
for~n=4,6,8,\cdots,
\end{equation}
\begin{equation}\label{k1}
DC_{1}^{-}(n,q)=Q(2n,q)\sigma_{n-1}Q(2n,q), \quad
for~n=1,3,5,\cdots,
\end{equation}
\begin{equation}\label{l1}
DC_{2}^{-}(n,q)=Q(2n,q)\sigma_{n-2}Q(2n,q), \quad for~n=3,5,7\cdots,
\end{equation}
\begin{equation}\label{m1}
DC_{3}^{-}(n,q)=\rho Q(2n,q)\sigma_{n-2}Q(2n,q), \quad
for~n=3,5,7,\cdots,
\end{equation}
\begin{equation}\label{n1}
DC_{4}^{-}(n,q)=\rho Q(2n,q)\sigma_{n-3}Q(2n,q), \quad
for~n=3,5,7,\cdots.
\end{equation}
Then, from (\ref{f1}), we have:
\begin{equation}\label{o1}
N_i^{\pm}(n,q)=|DC_i^{\pm}(n,q)|=A_i^{\pm}(n,q)B_i^{\pm}(n,q), ~for~
i=1,2,3,4
\end{equation}
(cf. (\ref{c})-(\ref{r})).

Unless otherwise stated, from now on, we will agree that anything
related to $DC_{1}^{+}(n,q)$, $DC_{2}^{+}(n,q)$  and
$DC_{3}^{+}(n,q)$ are defined for $n=2,4,6,\cdots,$ anything related
to $DC_{4}^{+}(n,q)$ is defined for $n=4,6,8,\cdots,$ anything
related to $DC_{1}^{-}(n,q)$ is defined for $n=1,3,5,\cdots$, and
anything related to $DC_{2}^{-}(n,q)$, $DC_{3}^{-}(n,q)$, and
$DC_{4}^{-}(n,q)$ are defined for $n=3,5,7,\cdots$.


\section{Exponential sums over double cosets of  $O^{-}(2n,q)$}
The following notations will be employed throughout this paper.
\begin{equation*}
\begin{split}
&tr(x)=x+x^{3}+\cdots+x^{3^{r-1}} ~the~trace~function~
\mathbb{F}_{q}\rightarrow\mathbb{F}_{3},\\
&\lambda_{0}(x)=e^{2\pi i x/3}~the ~canonical ~additive ~character
~of ~\mathbb{F}_{3},\\
&\lambda (x)=e^{2\pi i tr(x)/3}~the
~canonical ~additive ~character~ of~ \mathbb{F}_{q}.
\end{split}
\end{equation*}

Then any nontrivial additive character $\psi$ of $\mathbb{F}_{q}$ is
given by $\psi(x)=\lambda(ax)$, for a unique $a\in
\mathbb{F}_{q}^{*}$.

For any nontrivial additive character $\psi$ of $\mathbb{F}_q$ and
$a\in\mathbb{F}_{q}^{*}$, the Kloosterman sum $K_{GL(t,q)}(\psi;a)$
for $GL(t,q)$ is defined as
\begin{align*}
K_{GL(t,q)}(\psi;a)=\sum_{w\in GL(t,q)}\psi(Trw+aTrw^{-1}).
\end{align*}
Notice that, for $t=1$, $K_{GL(1,q)}(\psi;a)$ denotes the
Kloosterman sum $K(\psi;a)$.

In \cite{D3}, it is shown that $K_{GL(t,q)}(\psi;a)$ satisfies the
following recursive relation: for integers $t\geq2$,
$a\in\mathbb{F}_{q}^{*}$,
\begin{equation}\label{p1}
K_{GL(t,q)}(\psi;a)=q^{t-1}K_{GL(t-1,q)}(\psi;a)K(\psi;a)+q^{2t-2}(q^{t-1}-1)K_{GL(t-2,q)}(\psi;a),
\end{equation}
where we understand that $K_{GL(0,q)}(\psi,a)=1$.

\begin{proposition}\label{B}$($\cite{D1}$)$
Let $\psi$ be a nontrivial additive character of $\mathbb{F}_q$. For
each positive integer $r$, let $\Omega_r$ be the set of all $r\times
r$ nonsingular symmetric matrices over $F_q$. Then, with
$\delta_\epsilon$ as in (\ref{z}), we have
\begin{align*}
\begin{split}
b_r(\psi)&=\sum_{B\in\Omega_r}\sum_{h\in\mathbb{F}_{q}^{r\times 2}}\psi(Tr\delta_\epsilon {^t}hBh)\\
&=
\begin{cases}
q^{r(r+6)/4}\prod_{j=1}^{r/2}(q^{2j-1}-1),& \text {for r even,}\\
-q^{(r^2+4r-1)/4}\prod_{j=1}^{(r+1)/2}(q^{2j-1}-1),& \text {for r
odd.}
\end{cases}
\end{split}
\end{align*}
\end{proposition}

\begin{proposition}\label{C}$($\cite{D2}$)$
Let $\psi$ be a nontrivial additive character of $\mathbb{F}_q$.
Then
\item
$(1)$
\begin{align*}
 \sum_{w\in SO^-(2,q)}\psi(Trw)=-K(\psi;1),
\end{align*}
$(2)$
\begin{align*}
 \sum_{w\in SO^-(2,q)}\psi(Tr\delta_1w)=q+1,
\end{align*}
$(3)$
\begin{align*}
 \sum_{i\in O^-(2,q)}\psi(Trw)=-K(\psi;1)+q+1~(cf.(\ref{a1})),
\end{align*}
where
\begin{align*}
\delta_1=\begin{bmatrix}
           1 & 0 \\
           0 & -1 \\
         \end{bmatrix}
.
\end{align*}
\end{proposition}

Also, from Section 6 of \cite{D1}, it is shown that Gauss sum for
$O^-(2n,q)$, with $\psi$  a nontrivial additive character of
$\mathbb{F}_q$, is given by:
\begin{align*}
\begin{split}
\sum_{w\in O^-(2n,q)}\psi(Trw)&=\sum_{r=0}^{n-1}\sum_{w\in P\sigma_r Q}^{}\psi(Trw)\\
                              &=\sum_{r=0}^{n-1}\sum_{w\in Q\sigma_r Q}^{}\psi(Trw)
                              +\sum_{r=0}^{n-1}\sum_{w\in
                              \rho Q\sigma_rQ}^{}\psi(Trw) \quad (cf.~(\ref {b1}),(\ref
                              {c1})),
\end{split}
\end{align*}
with
\begin{align*}
\begin{split}
\sum_{w\in Q\sigma_r Q}^{}\psi(Trw)&= |B_r\setminus Q|\sum_{w \in
Q}\psi(Trw\sigma_r)\\                &=q^{(n-1)(n+2)/2}\sum_{i\in
                             SO^{-}(2,q)}\psi(Tri)\\&\times |B_r\setminus
                             Q|q^{r(n-r-3)}b_r(\psi)K_{GL(n-1-r,q)}(\psi;1),
\end{split}
\end{align*}
\begin{align*}
\begin{split}
\sum_{w\in \rho Q\sigma_r Q}^{}\psi(Trw)&= |B_r\setminus Q|\sum_{w
\in Q}\psi(Tr\rho w\sigma_r)\\ &=q^{(n-1)(n+2)/2}\sum_{i\in
                             SO^{-}(2,q)}\psi(Tr\delta_1 i)\\&\times |B_r\setminus
                             Q|q^{r(n-r-3)}b_r(\psi)K_{GL(n-1-r,q)}(\psi;1).
\end{split}
\end{align*}
Here one uses $(\ref {d1})$ and the fact that $\rho^{-1}w\rho \in
Q$, for all $w\in Q$.

Now, we see from $(\ref{e1})$ and Propositions 2 and 3 that, for
each $r$ with $0\leq r \leq n-1$,

\begin{equation}\label{q1}
\begin{split}
\sum_{w \in Q\sigma_r Q}\psi(Trw)=q^{(n-1)(n+2)/2} {\Big[
                                                      \begin{array}{c}
                                                        n-1 \\
                                                        r \\
                                                      \end{array}
                                                    \Big]
}_q
K(\psi;1)K_{GL(n-1-r,q)}(\psi;1)\\ \quad \quad \quad \quad \times
\begin{cases}
-q^{rn- \frac{1}{4}r^2}\prod_{j=1}^{r/2}(q^{2j-1}-1),& \text {for t even,}\\
q^{rn- \frac{1}{4}(r+1)^2}\prod_{j=1}^{(r+1)/2}(q^{2j-1}-1),& \text
{for t odd,}
\end{cases}
\end{split}
\end{equation}

\begin{equation}\label{r1}
\begin{split}
\sum_{w \in \rho Q\sigma_r Q}\psi(Trw)=(q+1)q^{(n-1)(n+2)/2} {\Big[
                                                      \begin{array}{c}
                                                        n-1 \\
                                                        r \\
                                                      \end{array}
                                                    \Big]}_q
K_{GL(n-1-r,q)}(\psi;1)\\ \quad \quad \quad \quad \times
\begin{cases}
q^{rn- \frac{1}{4}r^2}\prod_{j=1}^{r/2}(q^{2j-1}-1),& \text {for t even,}\\
-q^{rn- \frac{1}{4}(r+1)^2}\prod_{j=1}^{(r+1)/2}(q^{2j-1}-1),& \text
{for t odd.}
\end{cases}
\end{split}
\end{equation}
 For our purposes, we need the following special cases of exponential
 sums in $(\ref {q1})$ and $(\ref {r1})$. We state them separately
 as a theorem.

\begin{theorem}\label{D}
Let $\psi$ be a nontrivial additive character of $\mathbb{F}_q$.
Then, in the notations of $(\ref {c})$, $(\ref {e})$, $(\ref {g})$,
$(\ref {i})$, $(\ref {k})$, $(\ref {m})$, $(\ref {o})$, and $(\ref
{q})$, we have
\begin{align*}
\begin{split}
\sum_{w\in DC_{i}^{\pm}(n,q)}\psi(Trw)&={\pm}A_{i}^{\pm}(n,q)K(\psi;1),\quad \text{~for~} i=1,3,\\
\sum_{w\in DC_{2}^{\pm}(n,q)}\psi(Trw)&={\pm}(-1)A_{2}^{\pm}(n,q)K(\psi;1)^2, \\
\sum_{w\in
DC_{4}^{\pm}(n,q)}\psi(Trw)&={\pm}(-1)q^{-1}A_{4}^{\pm}(n,q)K_{GL(2,q)}(\psi;1)\\
                            &={\pm}(-1)A_{4}^{\pm}(n,q)(K(\psi;1)^2+q^2-q)\\
\end{split}
\end{align*}
(cf. (\ref{g1})-(\ref{n1}),(\ref{p1})).
\end{theorem}

\begin{corollary}\label{E}
Let $\lambda$ be the canonical additive character of $\mathbb{F}_q$,
and let $a\in \mathbb{F}_q^{*}$. Then we have
\begin{align}\label{s1}
\begin{split}
\sum_{w\in DC_{i}^{\pm}(n,q)}\lambda(aTrw)&={\pm}A_{i}^{\pm}(n,q)K(\lambda;a^2),
\quad \text{~for~} i=1,3,\\
\end{split}
\end{align}
\begin{align}\label{t1}
\begin{split}
\sum_{w\in DC_{2}^{\pm}(n,q)}\lambda(aTrw)&={\pm}(-1)A_{2}^{\pm}(n,q)K(\lambda;a^2)^2, \\
\end{split}
\end{align}
\begin{align}\label{u1}
\begin{split}
\sum_{w\in DC_{4}^{\pm}(n,q)}\lambda(aTrw)
&={\pm}(-1)A_{4}^{\pm}(n,q)(K(\lambda;a^2)^2+q^2-q).
\end{split}
\end{align}
\end{corollary}

\begin{proposition}\label{F} $($$[7,(5.3$-$5)]$$)$
Let $\lambda$ be the canonical additive character of $\mathbb{F}_q$,
 $m\in \mathbb{Z}_{\geq 0}$, $\beta\in \mathbb{F}_{q}^{}$. Then
\begin{align}\label{v1}
\begin{split}
\sum_{a\in \mathbb{F}_{q}^{*}}\lambda(-a\beta)K(\lambda;a^2)^m=
q\delta(m,q;\beta)-(q-1)^m,
\end{split}
\end{align}
where, for $m\geq1$,
\begin{align}\label{w1}
\begin{split}
\delta(m,q;\beta)=|\{(\alpha_{1},\cdots,\alpha_{m})\in
(\mathbb{F}_{q}^{*})^m|\alpha_{1}+\alpha_{1}^{-1}+\cdots+\alpha_{m}
+\alpha_{m}^{-1}=\beta\}|,
\end{split}
\end{align}
and
\begin{align*}
\begin{split}
\delta(0,q;\beta)=
\begin{cases}
1, ~\text{if}~ \beta=0,\\ 0, ~\text{otherwise}.
\end{cases}
\end{split}
\end{align*}
\end{proposition}

\begin{remark}\label{G}
Here one notes that
\begin{equation}\label{x1}
\begin{split}
\delta(1,q;\beta)&=|\{x\in \mathbb {F}_{q}|x^2-\beta x+1=0\}|\\ &=
\begin{cases}
2, ~\text {~if~} \beta^2-1 \neq 0  ~\text {is a square},\\
1, ~\text {~if~} \beta = {\pm 1},\\
0, ~\text {~if~} \beta^2-1 ~\text {is a nonsquare}.
\end{cases}
\end{split}
\end{equation}
\end{remark}

In the following lemma, $q$ is not just a power of 3 but a power of
any prime.

\begin{lemma}\label{H}
For any $\beta$,
\begin{equation}\label{y1}
\delta(2,q;\beta)\leq
\begin{cases}
2q-4, ~\text {~if~char~$\mathbb {F}_{q}\neq 2$},\\
2q-3, ~\text {~if~char~$\mathbb {F}_{q}= 2$}.
\end{cases}
\end{equation}
\end{lemma}
\begin{proof}
Firstly, we show that $\delta(2,q;\beta) \leq \delta(2,q;0)$ for any
$\beta$. Observe that
\begin{align*}
\begin{split}
\delta(2,q;\beta)=|\{(\alpha_{1},\alpha_{2})\in
\mathbb{F}_{q}^{2}|\alpha_{1}-\alpha_{2}+\alpha_{1}^{-1}-\alpha_{2}^{-1}=\beta\}|
\end{split}
\end{align*}
Then, borrowing an idea from \cite {SP1}, we have
\begin{align*}
\begin{split}
\delta(2,q;\beta)
&=q^{-1}\sum_{\alpha \in \mathbb {F}_q}\lambda(-\alpha\beta)
\sum_{{\alpha_1}^{}\in \mathbb {F}_{q}^{*}}\lambda(\alpha(\alpha_{1}^{}+\alpha_{1}^{-1}))
\sum_{{\alpha_2}^{}\in \mathbb {F}_{q}^{*}}\lambda(-\alpha(\alpha_{2}^{}+\alpha_{2}^{-1}))\\
&=q^{-1}\sum_{\alpha \in \mathbb {F}_q}\lambda(-\alpha\beta)
|\sum_{x\in \mathbb {F}_{q}^{*}}\lambda(\alpha(x_{}^{}+x_{}^{-1}))|^2\\
&\leq q^{-1}\sum_{\alpha \in \mathbb {F}_q}\sum_{x \in \mathbb
{F}_q^{*}}
|\lambda(\alpha(x_{}^{}+x_{}^{-1}))|^2\\
&=\delta(2,q;0).
\end{split}
\end{align*}
Here, for any prime power $q$,  $\lambda$  is the canonical additive
character of $\mathbb {F}_{q}$.

Secondly, we show that
\begin{equation}\label{z1}
\delta(2,q;0)=
\begin{cases}
2q-4, ~\text {~if~char~$\mathbb {F}_{q}\neq 2$},\\
2q-3, ~\text {~if~char~$\mathbb {F}_{q}= 2$}.
\end{cases}
\end{equation}
We see, by multiplying the equation
$\alpha_{1}^{}+\alpha_{2}^{}+\alpha_{1}^{-1}+\alpha_{2}^{-1}=0$ by
$\alpha_{1}^{}\alpha_{2}^{}$, that
\begin{align}\label{a2}
\begin{split}
\delta(2,q;0)=|\{(\alpha_{1},\alpha_{2})\in
\mathbb{F}_{q}^{2}|(\alpha_{1}\alpha_{2}+1)(\alpha_{1^{}}+\alpha_{2}^{})=0\}|-1,
\end{split}
\end{align}
and
\begin{align}\label{b2}
\begin{split}
\{(\alpha_{1},\alpha_{2})\in
\mathbb{F}_{q}^{2}|(\alpha_{1}\alpha_{2}+1)(\alpha_{1^{}}+\alpha_{2}^{})=0\}=A\cup
B,
\end{split}
\end{align}
with
\begin{align}\label{c2}
\begin{split}
A=\{(\alpha_{1},\alpha_{2})\in
\mathbb{F}_{q}^{2}|\alpha_{1}\alpha_{2}+1=0\},
B=\{(\alpha_{1},\alpha_{2})\in
\mathbb{F}_{q}^{2}|\alpha_{1}+\alpha_{2}=0\}.
\end{split}
\end{align}
Note here that $|A|=q-1$, and $|B|=q$.

Further, $A\cap B = \{{\pm}(1,-1)\}$, so that
\begin{equation}\label{d2}
|A\cap B|=
\begin{cases}
2, ~\text {~if~char~$\mathbb {F}_{q}\neq 2$},\\
1, ~\text {~if~char~$\mathbb {F}_{q}= 2$}.
\end{cases}
\end{equation}
From (\ref {a2})-(\ref {d2}), we get the result in (\ref {z1}).
\end{proof}

\begin{remark}\label{I}
We have shown in \cite{D7} that, for $char ~\mathbb {F}_{q}=2$,
\begin{equation*}
\delta(2,q;\beta)=
\begin{cases}
2q-3, ~\text {~if~} \beta=0,\\
K(\lambda;\beta^{-1})+q-3, ~\text {~if~} \beta \neq 0.
\end{cases}
\end{equation*}
For any integer $r$ with $0 \leq r \leq n-1$, and each $\beta \in
\mathbb {F}_{q}$, we let
\begin{equation*}
\begin{split}
&N_{Q\sigma_{r}Q}(\beta)=|\{w\in Q\sigma_{r}Q|Trw=\beta\}|,\\
&N_{\rho Q\sigma_{r}Q}(\beta)=|\{w\in \rho
Q\sigma_{r}Q|Trw=\beta\}|.
\end{split}
\end{equation*}
Then it is easy to see that
\begin{equation*}
\begin{split}
&qN_{Q\sigma_{r}Q}(\beta)=|Q\sigma_{r}Q|+\sum_{a \in
\mathbb{F}_{q}^{*}}\lambda(-a\beta)
\sum_{w \in {Q\sigma_rQ}}\lambda(aTrw),\\
&qN_{\rho Q\sigma_{r}Q}(\beta)=|\rho Q\sigma_{r}Q|+\sum_{a \in
\mathbb{F}_{q}^{*}}\lambda(-a\beta) \sum_{w \in {\rho
Q\sigma_rQ}}\lambda(aTrw).
\end{split}
\end{equation*}
\end{remark}
Now, from (\ref {g1})-(\ref {o1}) and (\ref {s1})-(\ref {v1}), we
have the following result.

\begin{proposition}\label{J}
$(1)$ For $~i=1,3,~$
\begin{equation}\label{e2}
\begin{split}
N_{DC_{i}^{\pm}(n,q)}(\beta)&=q^{-1}A_{i}^{\pm}(n,q)B_{i}^{\pm}(n,q)
{\pm}q^{-1}A_{i}^{\pm}(n,q)(q\delta(1,q;\beta)-q+1)\\
&=q^{-1}A_{i}^{\pm}(n,q)B_{i}^{\pm}(n,q)
{\pm}q^{-1}A_{i}^{\pm}(n,q)\\
&\quad\quad\times
\begin{cases}
q+1,& \text {if $\beta^2-1\neq 0$ is a square,}\\
1,& \text {if $\beta={\pm}1$},\\
-q+1,& \text {if $\beta^2-1$ is a nonsquare.}
\end{cases}
\end{split}
\end{equation}
$(2)$  $\qquad$ $N_{DC_{2}^{\pm}(n,q)}(\beta)$
\begin{equation}\label{f2}
\begin{split}
=q^{-1}A_{2}^{\pm}(n,q)B_{2}^{\pm}(n,q)
{\pm}(-1)q^{-1}A_{2}^{\pm}(n,q)\{q\delta(2,q;\beta)-(q-1)^2\}.
\end{split}
\end{equation}
$(3)$ $\qquad$
$N_{DC_{4}^{\pm}(n,q)}(\beta)=q^{-1}A_{4}^{\pm}(n,q)B_{4}^{\pm}(n,q)
{\pm}(-1)q^{-1}A_{4}^{\pm}(n,q)\\ $
\begin{equation}\label{g2}
\begin{split}
&\qquad\qquad\qquad\times
\begin{cases}
q\delta(2,q;\beta)-2q^2+3q-1,& \text {if $\beta\neq0$,}\\
q\delta(2,q;\beta)+q^3-3q^2+3q-1,& \text {if $\beta=0$.}
\end{cases}
\end{split}
\end{equation}
Here
$\delta(2,q;\beta)=|\{(\alpha_{1}^{},\alpha_{2}^{})\in(\mathbb{F}_{q}^{*})^2|\alpha_{1}^{}+
\alpha_{1}^{-1}+\alpha_{2}^{}+\alpha_{2}^{-1}=\beta\}|$.
\end{proposition}

\begin{corollary}\label{K}
\item\label{Ka}
$(1)$ For all even $n\geq 2$ and all $q$,
$N_{DC_{i}^{+}(n,q)}(\beta)>0$, for all $\beta$ and $i=1,2.$
\item\label{Kb}
$(2)$ For all even $n\geq 4$ and all $q$,
$N_{DC_{3}^{+}(n,q)}(\beta)>0$, for all $\beta$; for $n=2$  and  all
$q$,
\begin{equation}\label{h2}
\begin{split}
N_{DC_{3}^{+}(2,q)}(\beta)&=q^2(q+1)\delta(1,q;\beta)\\ &=
\begin{cases}
2q^3+2q^2, ~\text {~if~} \beta^2-1 \neq 0  ~\text {is a square},\\
q^3+q^2, ~\text {~if~} \beta = {\pm 1},\\
0, ~\text {~if~} \beta^2-1 ~\text {is a nonsquare}.
\end{cases}
\end{split}
\end{equation}
\item\label{Kc}
$(3)$ For all even $n\geq 4$ and all $q$,
$N_{DC_{4}^{+}(n,q)}(\beta)>0$, for all $\beta$.
\item\label{Kd}
$(4)$ For all odd $n\geq 3$ and all $q$,
$N_{DC_{1}^{-}(n,q)}(\beta)>0$, for all $\beta$; for $n=1$  and  all
$q$,
\begin{equation}\label{i2}
\begin{split}
N_{DC_{1}^{-}(1,q)}(\beta)&=2-\delta(1,q;\beta)\\ &=
\begin{cases}
0, ~\text {~if~} \beta^2-1 \neq 0  ~\text {is a square},\\
1, ~\text {~if~} \beta = {\pm 1},\\
2, ~\text {~if~} \beta^2-1 ~\text {is a nonsquare}.
\end{cases}
\end{split}
\end{equation}
\item\label{Ke}
$(5)$ For all odd $n\geq 3$ and all $q$,
$N_{DC_{i}^{-}(n,q)}(\beta)>0$, for all $\beta$ and  $i=2,3,4.$
\end{corollary}
\begin{proof}
It is tedious to check all the assertions in the statements. The
details are left to the reader, except that we make a comment on the
case of (1) with $i=2.$ We see that  $N_{DC_{2}^{+}(n,q)}(\beta)>0$,
for all $n\geq 4$ even and all $q.$ In addition,
\begin{equation*}
N_{DC_{2}^{+}(2,q)}(\beta)=q^2(2q-2-\delta(2,q;\beta))>0, ~ \text
{~in~view~of~(\ref{y1})}.
\end{equation*}
\end{proof}

\section{Construction of codes}
Here we will construct eight infinite families of ternary linear
codes $C(DC_{1}^{+}(n,q))$ of length $N_{1}^{+}(n,q))$,
$C(DC_{2}^{+}(n,q))$ of length $N_{2}^{+}(n,q)$,
$C(DC_{3}^{+}(n,q))$of length $N_{3}^{+}(n,q)$, for $n=2,4,6,\cdots$
and all $q$; $C(DC_{4}^{+}(n,q))$ of length $N_{4}^{+}(n,q)$, for
$n=4,6,8,\cdots $ and all $q$; $C(DC_{1}^{-}(n,q))$ of length
$N_{1}^{-}(n,q)$, for $n=1,3,5,\cdots $ and all $q$;,
$C(DC_{2}^{-}(n,q))$ of length $N_{2}^{-}(n,q)$,
$C(DC_{3}^{-}(n,q))$ of length $N_{3}^{-}(n,q)$,
$C(DC_{4}^{-}(n,q))$ of length $N_{4}^{-}(n,q)$, for $n=3,5,7,\cdots
$ and all $q$, respectively associated with the double cosets
$DC_{1}^{+}(n,q)$, $DC_{2}^{+}(n,q)$, $DC_{3}^{+}(n,q)$,
$DC_{4}^{+}(n,q)$, $DC_{1}^{-}(n,q)$, $DC_{2}^{-}(n,q)$,
$DC_{3}^{-}(n,q)$, $DC_{4}^{-}(n,q)$  (cf. (\ref{g1})-(\ref{n1})).
Let $g_{1},g_{2},\cdots ,g_{N_{i}^{\pm}(n,q)}$ be some fixed
orderings of the elements in $DC_{i}^{\pm}(n,q)$, for $i=1,2,3,4,$
by abuse of notations. Then we put
\begin{equation*}
v_{i}^{\pm}(n,q)=(Trg_{1},Trg_{2},\cdots,Trg_{N_{i}^{\pm}(n,q)})\in
\mathbb {F}_{q}^{N_{i}^{\pm}(n,q)}, ~\text {for}~ i= 1,2,3,4.
\end{equation*}
The ternary codes $C(DC_{1}^{+}(n,q))$, $C(DC_{2}^{+}(n,q))$,
$C(DC_{3}^{+}(n,q))$, $C(DC_{4}^{+}(n,q))$,
$C(DC_{1}^{-}(n,q))$,$C(DC_{2}^{-}(n,q))$, $C(DC_{3}^{-}(n,q))$, and
$C(DC_{4}^{-}(n,q))$ are defined as:
\begin{equation}\label{j2}
C(DC_{i}^{\pm}(n,q))=\{u \in \mathbb
{F}_{q}^{N_{i}^{\pm}(n,q)}|u\cdot v_{i}^{\pm}(n,q)=0\}, ~\text
{for}~ i= 1,2,3,4,
\end{equation}
where the dot denotes respectively the usual inner product in
$\mathbb {F}_{q}^{N_{i}^{\pm}(n,q)}$, for $i=1,2,3,4.$

  The following theorem of  Delsarte  is well-known.
\begin{theorem}\label{L}$($\cite{MS1}$)$
Let  $B$ be a linear code over $\mathbb {F}_{q}$. Then
\begin{equation*}
(B|_{\mathbb {F}_{3}})^{\perp}=tr(B^{\perp}).
\end{equation*}
\end{theorem}
In view of this theorem, the respective duals of the codes in
(\ref{j2}) are given by:
\begin{equation}\label{k2}
\begin{split}
&C(DC_{i}^{\pm}(n,q))^{\perp}=\{c_{i}^{\pm}(a)=c_{i}^{\pm}(a;n,q)
=(tr(aTrg_1),\cdots,tr(aTrg_{N_{i}^{\pm}(n,q)}))|a \in
\mathbb{F}_{q}\},\\
&\text {for}~ i=1,2,3,4.
\end{split}
\end{equation}
\begin{lemma}\label{M}
Let $\delta(m,q;\beta)$ be as in  (\ref{w1}),  and let $a \in
\mathbb{F}_{q}^{*}$. Then we have
\begin{equation}\label{l2}
\sum_{\beta \in
\mathbb{F}_{q}}\delta(m,q;\beta)\lambda(a\beta)=K(\lambda;a^{2})^m.
\end{equation}
\end{lemma}
\begin{proof}
The LHS of (\ref{l2}) is equal to
\begin{equation*}
\begin{split}
&\sum_{\beta \in \mathbb{F}_{q}}(q^{-1}\sum_{{x_{1},
\cdots,x_{m}}\in \mathbb{F}_{q}^{*}} \sum_{\alpha \in
\mathbb{F}_{q}}
\lambda(\alpha(x_{1}+\cdots+x_{m}+x_{1}^{-1}+\cdots+x_{m}^{-1}-\beta)))\lambda(a\beta)\\
&=q^{-1}\sum_{{x_{1},\cdots,x_{m}}\in \mathbb{F}_{q}^{*}}
\sum_{\alpha \in \mathbb{F}_{q}}
\lambda(\alpha(x_{1}+\cdots+x_{m}+x_{1}^{-1}+\cdots+x_{m}^{-1}))\sum_{\beta \in \mathbb{F}_{q}}\lambda(\beta(a-\alpha))\\
&=\sum_{{x_{1},\cdots,x_{m}}\in \mathbb{F}_{q}^{*}}
\lambda(a(x_{1}+\cdots+x_{m}+x_{1}^{-1}+\cdots+x_{m}^{-1}))\\
&=\sum_{{x_{1},\cdots,x_{m}}\in \mathbb{F}_{q}^{*}}
\lambda(x_{1}+\cdots+x_{m}+a^2x_{1}^{-1}+\cdots+a^2x_{m}^{-1})\\
&=K(\lambda;a^2)^m.
\end{split}
\end{equation*}
\end{proof}
\begin{theorem}\label{N}
$(1)$ The map $\mathbb{F}_{q}\rightarrow
C(DC_{i}^{+}(n,q))^{\perp}(a\mapsto c_{i}^{+}(a))$ (~for ~$i=1,2,3$)
is an $\mathbb{F}_{3}$-linear isomorphism for $n\geq 2$ even and all
$q$.\\
$(2)$ The map $\mathbb{F}_{q}\rightarrow
C(DC_{4}^{+}(n,q))^{\perp}(a\mapsto c_{4}^{+}(a))$ is an
$\mathbb{F}_{3}$-linear isomorphism for $n\geq 4$ even and all
$q$.\\
$(3)$ The map $\mathbb{F}_{q}\rightarrow
C(DC_{1}^{-}(n,q))^{\perp}(a\mapsto c_{1}^{-}(a))$ is an
$\mathbb{F}_{3}$-linear isomorphism for $n\geq 1$ odd and all $q$.\\
$(4)$ The map $\mathbb{F}_{q}\rightarrow
C(DC_{i}^{-}(n,q))^{\perp}(a\mapsto c_{i}^{-}(a))$ (~for ~$i=2,3,4$)
is an $\mathbb{F}_{3}$-linear isomorphism for $n\geq 3$ odd and all
$q$.\\
\end{theorem}
\begin{proof}
All maps are clearly $\mathbb{F}_{3}$-linear and surjective. Let for
t be in the kernel of map $\mathbb{F}_{q}\rightarrow
C(DC_{1}^{+}(n,q))^{\perp}(a\mapsto c_{1}^{+}(a))$. Then
$tr(aTrg)=0$, for all $g \in DC_{1}^{+}(n,q)$. Since, by Corollary
${\ref{K}(1)}$, $Tr:DC_{1}^{+}(n,q) \rightarrow \mathbb{F}_{q}$ is
surjective, and hence $tr(a\alpha)=0$, for all $\alpha \in
\mathbb{F}_{q}$. This implies that $ a=0$, since otherwise
$tr:\mathbb{F}_{q}\rightarrow \mathbb{F}_{3}$ would be the zero map.
This shows $i=1$ case of (1). All the other assertions can be
handled in the same way, except for $i=3$ and $n=2$ case of (1) and
$n=1$ case of (3).

Assume first that we are in the $i=3$ and $n=2$ case  of  (1). Let
$a$ be in the kernel of the map $\mathbb{F}_{q}\rightarrow
C(DC_{3}^{+}(2,q))^{\perp}(a\mapsto c_{3}^{+}(a))$ . Then
$tr(aTrg)=0$, for all $g \in DC_{3}^{+}(2,q).$ Suppose that $a\neq
0$. Then we would have
\begin{equation*}
\begin{split}
q^2(q^2-1)=|DC_{3}^{+}(2,q)|&=\sum_{g \in DC_{3}^{+}(2,q)}e^{2\pi itr(aTrg)/3}\\
&=\sum_{\beta \in \mathbb {F}_{q}}N_{DC_{3}^{+}(2,q)}(\beta)\lambda(a\beta)\\
&=q^2(q+1)\sum_{\beta \in \mathbb {F}_{q}}
\delta(1,q;\beta)\lambda(a\beta)~(cf.(\ref{h2}))\\
&=q^2(q+1)K(\lambda;a^2)~(cf.(\ref{l2})).
\end{split}
\end{equation*}
So, using Weil bound in (\ref{a}), we would get
\begin{equation*}
q-1=K(\lambda;a^2)\leq 2\sqrt{q}.
\end{equation*}
For $q\geq 9$, this is impossible, since $x-1>2\sqrt{x}$, for $x\geq
9$. So the map $\mathbb{F}_{q}\rightarrow
C(DC_{3}^{+}(2,q))^{\perp}(a\mapsto c_{3}^{+}(a))$ is an
$\mathbb{F}_{3}$ -linear isomorphism if $q\geq 9$. This is also true
for $q=3$. Indeed, if $a$ is in the kernel of the map, then
$aTrg=0$, for all $g\in DC_{3}^{+}(2,3)$. Here $Trg=1$ for a half of
elements $g\in DC_{3}^{+}(2,3)$, and $Trg=-1$, for the other half of
elements $g\in DC_{3}^{+}(2,3)$(cf. (\ref{h2})). So $a=0$.
  Assume
next that we are in the $n=1$ case of  (3).  Let $a$ be in the
kernel of the map $\mathbb{F}_{q}\rightarrow
C(DC_{1}^{-}(1,q))^{\perp}(a\mapsto c_{1}^{-}(a))$. Then
$tr(aTrg)=0$, for all $g\in DC_{1}^{-}(1,q)$. Assume that $a\neq 0$.
Then, again using Weil bound, we would have
\begin{equation*}
\begin{split}
q+1=|DC_{1}^{-}(1,q)|&=\sum_{g \in DC_{1}^{-}(1,q)}e^{2\pi itr(aTrg)/3}\\
&=\sum_{\beta \in \mathbb {F}_{q}}N_{DC_{1}^{-}(1,q)}(\beta)\lambda(a\beta)\\
&=\sum_{\beta \in \mathbb {F}_{q}}
(2-\delta(1,q;\beta))\lambda(a\beta)~(cf.(\ref{i2}))\\
&=-\sum_{\beta \in \mathbb {F}_{q}}
\delta(1,q;\beta)\lambda(a\beta)~(\ref{l2})\\
&=-K(\lambda ;a^2)\\
&\leq 2\sqrt{q}.
\end{split}
\end{equation*}
So we would get $q=1$, which is impossible.
\end{proof}

\section{Recursive formulas for power moments of Kloosterman sums}
Here we will be able to find, via Pless power moment identity,
infinite families of recursive formulas generating power moments of
Kloosterman sums with square arguments and even power moments of
those in terms of the frequencies of weights in
$C(DC_{i}^{\pm}(n,q))$, for $i=1,3$ and in $C(DC_{i}^{\pm}(n,q))$,
for $i=2,4,$ respectively.

\begin{theorem}\label{O}$($Pless power moment identity, \cite{MS1}$)$
Let  $B$ be an $q$-$ary$ $[n,k]$code, and let
$B_{i}(resp.~B_{i}^{\perp})$ denote the number of codewords of
weight $i$ in $B$ (resp. in $B_{}^{\perp}$). Then, for
$h=0,1,2,\cdots,$
\begin{equation}\label{m2}
\sum_{j=0}^{n}j^hB_j=\sum_{j=0}^{min\{n,h\}}(-1)^{j}B_{j}^{\perp}
\sum_{t=j}^{h}t!S(h,t)q^{k-t}(q-1)^{t-j}\binom{n-j}{n-t},
\end{equation}
where $S(h,t)$ is the Stirling number of the second kind defined in
(\ref{u}).
\end{theorem}
\noindent
\begin{lemma}\label{P}
Let
$c_{i}^{\pm}(a)=(tr(aTrg_1),\cdots,tr(aTrg_{N_{i}^{\pm}(n,q)}))\in
C(DC_{i}^{\pm}(n,q))^{\perp}$, for $a\in \mathbb{F}_{q}^{*}$, and
$i=1,2,3,4.$ Then their Hamming weights are expressed as follows:\\
$(1)$ $\qquad$$w(c_{i}^{\pm}(a))$
\begin{equation}\label{n2}
=\frac{2}{3}A_{i}^{\pm}(n,q)\{B_{i}^{\pm}(n,q)
{\pm}(-1)K(\lambda;a^2) \}, \text {~for~} i=1,3,~\qquad
\end{equation}
$(2)$ $\qquad$ $w(c_{2}^{\pm}(a))$
\begin{equation}\label{o2}
=\frac{2}{3}A_{2}^{\pm}(n,q)\{B_{2}^{\pm}(n,q) {\pm}K(\lambda;a^2)^2
\},\qquad\qquad\qquad\qquad\qquad
\end{equation}
$(3)$ $\qquad$ $w(c_{4}^{\pm}(a))$
\begin{equation}\label{p2}
~~=\frac{2}{3}A_{4}^{\pm}(n,q)\{B_{4}^{\pm}(n,q)
{\pm}(q^2-q+K(\lambda;a^2)^2) \}~(cf.(\ref{c})-(\ref{r})).
\end{equation}
\begin{proof}
\begin{align*}
\begin{split}
w(c_i^{\pm}(a))&=\sum_{j=1}^{N_i^{\pm}(n,q)}(1-{\frac{1}{3}}
\sum_{\alpha\in\mathbb{F}_3}\lambda_0(\alpha tr(aTrg_j)))\\
&=N_i^{\pm}(n,q)-{\frac{1}{3}}\sum_{\alpha\in\mathbb{F}_3}\sum_{w\in
DC_i^{\pm}(n,q)}\lambda(\alpha aTrw)\\
&={\frac{2}{3}}N_i^{\pm}(n,q)-{\frac{1}{3}}\sum_{\alpha\in\mathbb{F}_3^*}
\sum_{w\in DC_i^{\pm}(n,q)}\lambda(\alpha aTrw), \quad \text ~{for}~
i=1,2,3,4.
\end{split}
\end{align*}
Our results now follow from (\ref{o1}) and (\ref{s1})-(\ref{u1}).
\end{proof}
\end{lemma}

Let $u=(u_1,\cdots,u_{N_{i}^{\pm}(n,q)})\in \mathbb
{F}_{3}^{N_{i}^{\pm}(n,q)}$, for $i=1,2,3,4,$ with $\nu_{\beta}$ 1's
and $\mu_{\beta}$ 2's in the coordinate places where
$Tr(g_j)=\beta$, for each $\beta \in \mathbb {F}_{q}$. Then, from
the definition of the codes $C(DC_{i}^{\pm}(n,q))$(cf. (\ref{j2})),
we see that $u$ is a codeword with weight $j$ if and only if
$\sum_{\beta\in\mathbb{F}_q} \nu_\beta
+\sum_{\beta\in\mathbb{F}_q}\mu_\beta=j\quad and \quad
\sum_{\beta\in\mathbb{F}_q} \nu_\beta
\beta=\sum_{\beta\in\mathbb{F}_q}\mu_\beta \beta$ (an identity in
$\mathbb{F}_q$). As there are
$\prod_{\beta\in\mathbb{F}_q}\binom{N_{DC_{i}^{\pm}(n,q)}(\beta)}{\nu_{\beta},\mu_{\beta}}$
many such codewords with weight $j$, we obtain the following result.
\begin{proposition}\label{Q}
Let $\{C_{i,j}^{\pm}(n,q)\}_{j=0}^{N_{i}^{\pm}(n,q)}$ be the weight
distribution of $C(DC_{i}^{\pm}(n,q))$, for $i=1,2,3,4.$ Then we
have
\begin{equation}\label{q2}
C_{i,j}^{\pm}(n,q)=\sum\prod_{\beta \in \mathbb {F}_{q}}
\binom{N_{DC_{i}^{\pm}(n,q)}(\beta)}{\nu_{\beta},\mu_{\beta}}, \text
{~for~} 0\leq j \leq N_{i}^{\pm}(n,q), \text {~and~} i= 1,2,3,4,
\end{equation}
where the sum is over all the sets of nonnegative integers
$\{\nu_{\beta}\}_{\beta \in \mathbb {F}_{q}}$ and
$\{\mu_{\beta}\}_{\beta \in \mathbb {F}_{q}}$ satisfying
\begin{equation*}
\sum_{\beta\in\mathbb{F}_q} \nu_\beta
+\sum_{\beta\in\mathbb{F}_q}\mu_\beta=j,\quad \text {and} \quad
\sum_{\beta\in\mathbb{F}_q} \nu_\beta
\beta=\sum_{\beta\in\mathbb{F}_q}\mu_\beta \beta.
\end{equation*}
The formulas appearing in the next theorem and stated in (\ref {t}),
(\ref {w}), and (\ref {y}) follow  by applying the formula in (\ref
{q2}) to each $C(DC_{i}^{\pm}(n,q))$, using the explicit values of
$N_{DC_{i}^{\pm}(n,q)}(\beta)$ in (\ref {e2})-(\ref {g2}).
\end{proposition}
\begin{theorem}\label{R}
Let $\{C_{i,j}^{\pm}(n,q)\}_{j=0}^{N_{i}^{\pm}(n,q)}$ be the weight
distribution of $C(DC_{i}^{\pm}(n,q))$, for $i=1,2,3,4.$ Then we
have
\item \label{Ra}
$(1)$ For $i=1,3,$ and $j=0,\cdots, N_{i}^{\pm}(n,q)$,
\begin{equation*}
\begin{split}
C_{i,j}^{\pm}(n,q)=&\sum{\binom{q^{-1}A_{i}^{\pm}(n,q)(B_{i}^{\pm}(n,q)
{\pm}1)}{\nu_1,\mu_1}}{\binom{q^{-1}A_{i}^{\pm}(n,q)(B_{i}^{\pm}(n,q)
{\pm}1)}{\nu_{-1},\mu_{-1}}}\\
&\times \prod_{\beta^2-1\neq
0~square}{\binom{q^{-1}A_{i}^{\pm}(n,q)(B_{i}^{\pm}(n,q)
{\pm}(q+1))}{\nu_\beta,\mu_\beta}}\\
&\times
\prod_{\beta^2-1~nonsquare}{\binom{q^{-1}A_{i}^{\pm}(n,q)(B_{i}^{\pm}(n,q)
{\pm}(-q+1))}{\nu_\beta,\mu_\beta}},
\end{split}
\end{equation*}
where the sum is over all the sets of nonnegative integers
$\{\nu_{\beta}\}_{\beta \in \mathbb {F}_{q}}$ and
$\{\mu_{\beta}\}_{\beta \in \mathbb {F}_{q}}$ satisfying
\begin{equation*}
\sum_{\beta\in\mathbb{F}_q} \nu_\beta
+\sum_{\beta\in\mathbb{F}_q}\mu_\beta=j,\quad \text {and} \quad
\sum_{\beta\in\mathbb{F}_q} \nu_\beta
\beta=\sum_{\beta\in\mathbb{F}_q}\mu_\beta \beta.
\end{equation*}
\item \label{Rb}
$(2)$ For $j=0,\cdots, N_{2}^{\pm}(n,q)$,
\begin{equation*}
C_{2,j}^{\pm}(n,q)=\sum\prod_{\beta\in
\mathbb{F}_{q}}{\binom{q^{-1}A_{2}^{\pm}(n,q)(B_{2}^{\pm}(n,q)
{\pm}((q-1)^2-q\delta(2,q;\beta)))}{\nu_\beta,\mu_\beta}},
\end{equation*}
where the sum is over all the sets of nonnegative integers
$\{\nu_{\beta}\}_{\beta \in \mathbb {F}_{q}}$ and
$\{\mu_{\beta}\}_{\beta \in \mathbb {F}_{q}}$ satisfying
\begin{equation*}
\sum_{\beta\in\mathbb{F}_q} \nu_\beta
+\sum_{\beta\in\mathbb{F}_q}\mu_\beta=j,\quad \text {and} \quad
\sum_{\beta\in\mathbb{F}_q} \nu_\beta
\beta=\sum_{\beta\in\mathbb{F}_q}\mu_\beta \beta.
\end{equation*}
\item \label{Rc}
$(3)$ For  $j=0,\cdots, N_{4}^{\pm}(n,q)$,
\begin{equation*}
\begin{split}
C_{4,j}^{\pm}(n,q)=&\sum{\binom{q^{-1}A_{4}^{\pm}(n,q)(B_{4}^{\pm}(n,q)
{\pm}(-1)(q\delta(2,q;\beta)+(q-1)^3))}{\nu_0,\mu_0}}\\
&\times \prod_{\beta \neq
0}{\binom{q^{-1}A_{4}^{\pm}(n,q)(B_{4}^{\pm}(n,q)
{\pm}(-1)(q\delta(2,q;\beta)-2q^2+3q-1))}{\nu_\beta,\mu_\beta}},\\
\end{split}
\end{equation*}
where the sum is over all the sets of nonnegative integers
$\{\nu_{\beta}\}_{\beta \in \mathbb {F}_{q}}$ and
$\{\mu_{\beta}\}_{\beta \in \mathbb {F}_{q}}$ satisfying
\begin{equation*}
\sum_{\beta\in\mathbb{F}_q} \nu_\beta
+\sum_{\beta\in\mathbb{F}_q}\mu_\beta=j,\quad \text {and} \quad
\sum_{\beta\in\mathbb{F}_q} \nu_\beta
\beta=\sum_{\beta\in\mathbb{F}_q}\mu_\beta \beta.
\end{equation*}
\end{theorem}
Now, we apply the Pless power moment identity in (\ref {m2}) to
$C(DC_{i}^{\pm}(n,q))$, for $i=1,2,3,4,$ in order to get the results
in Theorem 1(cf. (\ref{s}), (\ref{t}), (\ref{v})-(\ref{y})) about
recursive formulas.

The left hand side of that identity in (\ref{m2}) is equal to
\begin{equation*}
\sum_{a\in \mathbb {F}_{q}^{*}}w(c_{i}^{\pm}(a))^{h},
\end{equation*}
with $w(c_{i}^{\pm}(a))$ given by (\ref{n2})-(\ref{p2}). We have,
for $i=1,3,$
\begin{equation}\label{r2}
\begin{split}
\sum_{a\in \mathbb {F}_{q}^{*}}w(c_{i}^{\pm}(a))^{h}
&=({\frac{2}{3}})^hA_{i}^{\pm}(n,q)^{h}\sum_{a\in \mathbb
{F}_{q}^{*}}\{B_{i}^{\pm}(n,q){\pm}(-1)K(\lambda;a^2)\}^h \\
&=2({\frac{2}{3}})^hA_{i}^{\pm}(n,q)^{h}\sum_{l=0}^{h}({\pm}(-1))^{l}\binom{h}{l}B_{i}^{\pm}(n,q)^{h-l}SK^{l}.
\end{split}
\end{equation}
Similarly, we have
\begin{equation}\label{s2}
\sum_{a\in \mathbb {F}_{q}^{*}}w(c_{2}^{\pm}(a))^{h}=
2({\frac{2}{3}})^hA_{2}^{\pm}(n,q)^{h}\sum_{l=0}^{h}(({\pm}1))^{l}\binom{h}{l}B_{2}^{\pm}(n,q)^{h-l}SK^{2l},
\end{equation}
\begin{equation}\label{t2}
\sum_{a\in \mathbb {F}_{q}^{*}}w(c_{4}^{\pm}(a))^{h}=
2({\frac{2}{3}})^hA_{4}^{\pm}(n,q)^{h}\sum_{l=0}^{h}({\pm}1)^{l}\binom{h}{l}\{B_{4}^{\pm}(n,q){\pm}(q^2-q)\}^{h-l}SK^{2l}.
\end{equation}
Here one has to separate the term corresponding to $l=h$ in
(\ref{r2})-(\ref{t2}), and notes $dim_{\mathbb
{F}_{3}}C(DC_{i}^{\pm}(n,q))^{\perp}=r$

\bibliographystyle{amsplain}

\end{document}